\documentclass{amsart}
\usepackage{mathrsfs,amssymb,mathrsfs, multirow,setspace}
\usepackage{amsmath}
\usepackage{listings}
\usepackage{pdfpages}
\usepackage{caption}
\usepackage[a4paper, total={7in, 9in}]{geometry}
\usepackage{enumerate}
\theoremstyle{plain}
\usepackage{graphicx}
\usepackage{mathtools}
\usepackage[utf8]{inputenc}

\newtheorem{theorem}{Theorem}[section]
\newtheorem{lemma}[theorem]{Lemma}

\newtheorem{corollary}[theorem]{Corollary}

\theoremstyle{definition}

\newtheorem{example}[theorem]{Example}

\theoremstyle{remark}
\newtheorem{remark}[theorem]{Remark}

\input xy
\xyoption{all}

\begin{document}
	\title[Strong metric dimension of the prime ideal sum graph of a commutative ring]{Strong metric dimension of the prime ideal sum graph of a commutative ring}

   \author[Praveen Mathil, Jitender Kumar]{Praveen Mathil$^{^1}$, Jitender Kumar$^{^{*1}}$, Reza Nikandish$^{^2}$}
   \address{$^1$Department of Mathematics, Birla Institute of Technology and Science   Pilani, Pilani-333031, India}
    \address{$^2$Department of Mathematics, Jundi-Shapur University of Technology, P.O. BOX 64615-334, Dezful, Iran}
 \email{maithilpraveen@gmail.com,  jitenderarora09@gmail.com, r.nikandish@ipm.ir}

\begin{abstract}
Let $R$ be a commutative ring with unity. The prime ideal sum graph of the ring $R$ is the simple undirected graph whose vertex set is the set of all nonzero proper ideals of $R$ and two distinct vertices $I$ and $J$ are adjacent if and only if $I + J$ is a prime ideal of $R$. In this paper, we obtain the strong metric dimension of the prime ideal sum graph for various classes of Artinian non-local commutative rings.
\end{abstract}
\keywords{Prime ideal sum graph, strong metric dimension, Artinian ring, reduced ring  \\ *  Corresponding author}
\maketitle

\section{Introduction and Preliminaries}
Metric dimension has many applications in robot navigation, image processing, combinatorial optimization, chemistry, network security and so on (see \cite{khuller1996landmarks,oellermann2007strong,slater1975leaves}). A close parameter to the metric dimension is the strong metric dimension of a graph. Seb\H{o} and Tannier \cite{sebHo2004metric} introduced the notion of the strong metric dimension of a graph and illustrated some applications in combinatorial searching. The problem of determining the strong metric dimension of a graph is NP-hard. Many researchers obtained the strong metric dimension for various classes of graphs, see \cite{kuziak2013strong, kuziak2015strong, rodriguez2014strong}. For more detail on the study of strong metric dimension of a graph, one can refer to \cite{kratica2014strong}.
Numerous applications of strong metric dimension and its significant background motivate algebraic graph theorists to determine the strong metric dimension of graphs associated with algebraic structures such as groups and rings (see \cite{ebrahimi2021strong,ma2018strong,ma2021strong,nikandish2021metric,nikandish2022strong,zhai2023metric}).

The prime ideal sum graph of a commutative ring was introduced by Saha \emph{et al.} \cite{saha2023prime}. The \emph{prime ideal sum graph} $\text{PIS}(R)$ of the ring $R$ is the simple undirected graph whose vertex set is the set of all nonzero proper ideals of $R$ and two distinct vertices $I$ and $J$ are adjacent if and only if $I + J$ is a prime ideal of $R$. In \cite{saha2023prime}, authors studied an interplay between graph-theoretic properties of $\text{PIS}(R)$ and algebraic properties of the ring $R$. They investigated the clique number, the chromatic number and the domination number of prime ideal sum graph $\text{PIS}(R)$. Embedding of the prime ideal sum graphs on various surfaces has been studied by Mathil et al. \cite{mathil2022embedding}. Recently, the metric dimension of the prime ideal sum graph of various classes of rings has been investigated in \cite{adlifard2023metric}. This paper aims to determine the strong metric dimension of the prime ideal sum graphs of commutative rings.

A \emph{graph} $\Gamma$ is an ordered pair $(V(\Gamma), E(\Gamma))$, where $V(\Gamma)$ is the set of vertices and $E(\Gamma)$ is the set of edges of $\Gamma$. Two distinct vertices $u, v \in V(\Gamma)$ are $\mathit{adjacent}$ in $\Gamma$, denoted by $u \sim v$ (or $(u,  v)$), if there is an edge between $u$ and $v$. Otherwise, we write as $u \nsim v$. Let $\Gamma$ be a graph. A graph $\Gamma' = (V(\Gamma'), E(\Gamma'))$ is said to be a \emph{subgraph} of $\Gamma$ if $V(\Gamma') \subseteq V(\Gamma)$ and $E(\Gamma') \subseteq E(\Gamma)$.
 A \emph{path} in a graph is a sequence of distinct vertices with the property that each vertex in the sequence is adjacent to the next vertex of it. The distance $d(x,y)$ between any two vertices $x$ and $y$ of $\Gamma$ is the number of edges in a shortest path between $x$ and $y$. The \emph{diameter} \rm{diam}$(\Gamma)$ of a connected graph $\Gamma$ is the maximum of the distances between vertices in $\Gamma$. A graph $\Gamma$ is said to be \emph{complete} if any two vertices are adjacent in $\Gamma$. A complete subgraph of the graph $\Gamma$ is said to be the \emph{clique}. The \emph{clique number} $\omega(\Gamma)$ of $\Gamma$ is the cardinality of a largest clique in $\Gamma$. 

Let $\Gamma$ be a graph. A vertex $z$ resolves two distinct vertices $x$ and $y$ if $d(x, z) \neq d(y, z)$. A subset $W$ of $V(\Gamma)$ is a \emph{resolving set} of $\Gamma$ if every pair of distinct vertices of $\Gamma$ is resolved by some vertex in $W$. A vertex $w$ in $\Gamma$ strongly resolves two vertices $u$ and $v$ if there exists a shortest path between $u$ and $w$ containing $v$, or there exists a shortest path between $v$ and $w$ containing $u$. A subset $S$ of vertex set of $\Gamma$ is a \emph{strong resolving set} of $\Gamma$ if every two distinct vertices of $\Gamma$ are strongly resolved by some vertex of $S$. The \emph{strong metric dimension} $\text{sdim}(\Gamma)$ of $\Gamma$  is the smallest cardinality of a strong resolving set in $\Gamma$.

Let $R$ be a commutative ring with unity. The prime ideal sum graph $\text{PIS}(R)$ of the ring $R$ is the simple undirected graph whose vertex set is the set of all nonzero proper ideals of $R$ and two distinct vertices $I$ and $J$ are adjacent if and only if $I + J$ is a prime ideal of $R$. Throughout the paper, the ring $R$ is an Artinian non-local commutative ring with unity and $F_i$ denotes a finite field. For basic definitions of ring theory, we refer the reader to  \cite{atiyah1969introduction}. A ring $R$ is said to be \emph{local} if it has a unique maximal ideal $\mathcal{M}$, we abbreviate it as $(R, \mathcal{M})$. By $\mathcal{I}^*(R)$, we mean the set of non-trivial proper ideals of $R$. The \emph{nilpotent index} $\eta(I)$ of an ideal $I$ of $R$ is the smallest positive integer $n$ such that $I^n = 0$. 
 By the structural theorem \cite{atiyah1969introduction}, an Artinian non-local commutative ring $R$ is uniquely (up to isomorphism) a finite direct product of  local rings $R_i$ that is $R \cong R_1 \times R_2 \times \cdots \times R_n$, where $n \geq 2$.

\section{Strong metric dimension of the graph $\text{PIS}(R)$ }

In this section, we investigate the strong metric dimension of the prime ideal sum graph $\text{PIS}(R)$ of an Artinian non-local commutative ring $R$. The graph $\text{PIS}(R)$ is disconnected if and only if $R$ is a direct product of two fields (see {\cite[Lemma 2.6]{saha2023prime}}). Thus, we avoid such situation in this paper. Note that, for a connected graph, every strongly resolving set is a resolving set. Thus, the following lemma is the direct consequence of {\cite[Lemma 2.1]{adlifard2023metric}}.

\begin{lemma}
    Let $R$ be a ring. Then $\textnormal{sdim}(\textnormal{PIS}(R))$ is finite if and only if $R$ has only finitely many ideals.
\end{lemma}

By $N(x)$, we mean the neighborhood of the vertex $x$ and $N[x]= N(x) \cup \{ x\}$. For $x,y \in \Gamma$, define a relation $\approx$ such that $x \approx y$ if and only if $N[x]= N[y]$. Observe that $\approx$ is an equivalence relation over $V(\Gamma)$. Let $U(\Gamma)$ be the set of distinct representatives of the equivalence relation $\approx$. The reduced graph $R_\Gamma$ of $\Gamma$ has the vertex set $U(\Gamma)$ and two distinct vertices of $R_\Gamma$ are adjacent if they are adjacent in $\Gamma$. The following theorem is useful to obtain the strong metric dimension of the graph $\Gamma$.
 
\begin{theorem}{\cite[Theorem 2.2]{ma2018strong}}\label{reducedgraphsdim}
Let $\Gamma$ be a connected graph with diameter two. Then 
\[ \ \text{sdim}(\Gamma) = |V(\Gamma)|- \omega(R_\Gamma).    \]
\end{theorem}

The following examples illustrate the Theorem \ref{reducedgraphsdim}.

\begin{example}
    Let $R \cong F_1 \times F_2 \times F_3$, where $F_i$ ($1 \le i \le 3$) is a finite field. By Figure \ref{example1}, observe that 
    \[ S = \{F_1 \times F_2 \times (0), F_1 \times (0) \times F_3, (0) \times F_2 \times F_3 \} \]
    is a minimum strong resolving set of $\textnormal{PIS}(R)$. Thus, $\text{sdim}(\textnormal{PIS}(R)) = 3$. On the other hand, note that $\textnormal{PIS}(R)$ is the reduced graph of itself and $\omega(R_{\textnormal{PIS}(R)})=3$. Consequently, by Theorem \ref{reducedgraphsdim}, $\text{sdim}(\textnormal{PIS}(R)) =6-3 = 3$.
    \begin{figure}[h!]
\centering
\includegraphics[width=0.4 \textwidth]{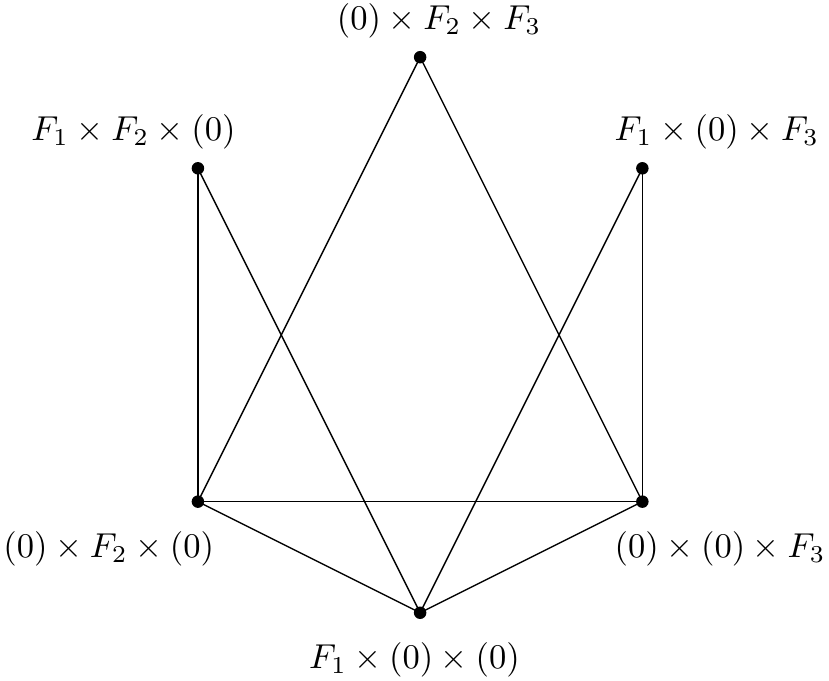}
\caption{$\text{PIS}(F_1 \times F_2 \times F_3)$}
\label{example1}
\end{figure}
\end{example}

\begin{example}
Let $R \cong \mathbb{Z}_4 \times \mathbb{Z}_9$. By Figure \ref{example2}, observe that $S = \{(2) \times (1), (0) \times (1), (2) \times (3), (2) \times (0) \}$ is a minimum strong resolving set of $\textnormal{PIS}(R)$. Thus, $\text{sdim}(\textnormal{PIS}(R)) = 4$. On the other hand, note that $R_{\textnormal{PIS}(\mathbb{Z}_4 \times \mathbb{Z}_9)} = \textnormal{PIS}(\mathbb{Z}_4 \times \mathbb{Z}_9)$ and $\omega(R_{\textnormal{PIS}(\mathbb{Z}_4 \times \mathbb{Z}_9)})=3$. Consequently, by Theorem \ref{reducedgraphsdim}, $\text{sdim}(\textnormal{PIS}(R)) =7-3 = 4$.
   \begin{figure}[h!]
\centering
\includegraphics[width=0.35 \textwidth]{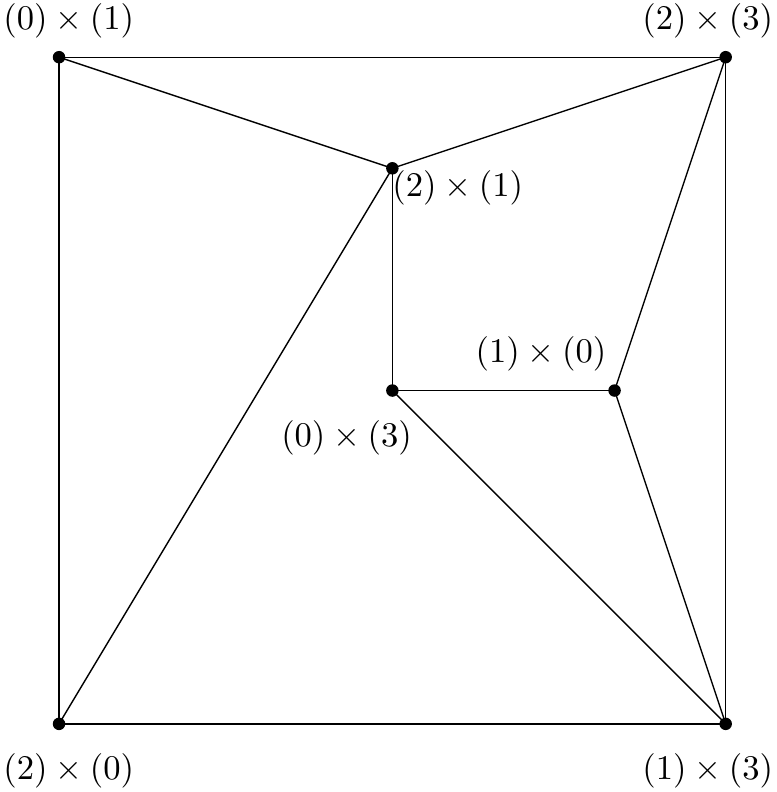}
\caption{$\text{PIS}(\mathbb{Z}_4 \times \mathbb{Z}_9)$}
\label{example2}
\end{figure}
\end{example}

\begin{example}
Let $R \cong \mathbb{Z}_8 \times \mathbb{Z}_{27}$. By Figure \ref{example3}, observe that
\[S = \{ (0) \times (9), (4) \times (0), (0) \times (1), (2) \times (1), (4) \times (9), (2) \times (0), (1) \times (0), (2) \times (9), (2) \times (3), (4) \times (1), (0) \times (3) \} \]
is a minimum resolving set of $\textnormal{PIS}(\mathbb{Z}_8 \times \mathbb{Z}_{27})$. Thus, $\text{sdim}(\textnormal{PIS}(R)) = 11$. On the other hand, note that $R_{\textnormal{PIS}( \mathbb{Z}_8 \times \mathbb{Z}_{27})} = \textnormal{PIS}( \mathbb{Z}_8 \times \mathbb{Z}_{27})$ and $\omega(R_{\textnormal{PIS}( \mathbb{Z}_8 \times \mathbb{Z}_{27})})=3$. Consequently, by Theorem \ref{reducedgraphsdim}, $\text{sdim}(\textnormal{PIS}(R)) =14-3 = 11$.
   \begin{figure}[h!]
\centering
\includegraphics[width=0.6 \textwidth]{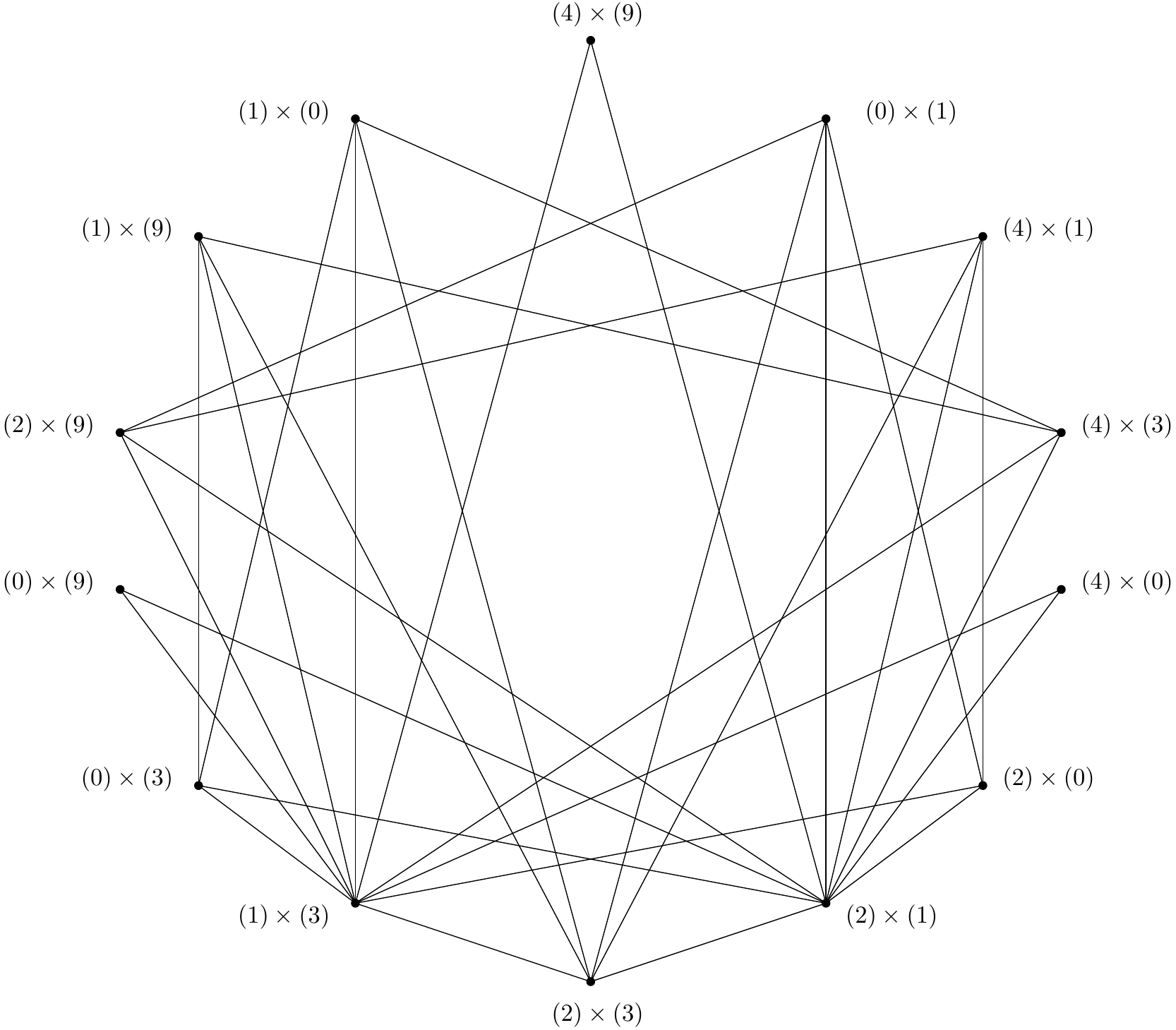}
\caption{$\text{PIS}(\mathbb{Z}_8 \times \mathbb{Z}_{27})$}
\label{example3}
\end{figure}
\end{example}

\begin{theorem}\label{daimeter}
    Let $R \cong R_1 \times R_2 \times \cdots \times R_n$  $(n \ge 2)$ be an Artinian non-local commutative ring. Then $\textnormal{diam}(\textnormal{PIS}(R)) =2$. 
\end{theorem}

\begin{proof}
Let $R \cong R_1 \times R_2 \times \cdots \times R_n$  $(n \ge 2)$, where each $R_i$ is a local ring with maximal ideal $\mathcal{M}_i$. Let $I = I_1 \times I_2 \times \cdots \times I_n$ and $J = J_1 \times J_2 \times \cdots \times J_n$ be any two distinct vertices of $\textnormal{PIS}(R)$ such that $I \nsim J$. If there exists $ k \in \{ 1,2, \ldots, n \}$ such that $ I_k , J_k \neq R_k$, then choose $I' = R_1 \times R_2 \times \cdots \times R_{k-1} \times \mathcal{M}_k \times R_{k+1} \times \cdots \times R_n$. Then $ I \sim I' \sim J$. If there does not exists any $k \in \{ 1,2, \ldots ,n\}$ such that both $I_k$ and $J_k$ are not equal to $R_k$, then there exist distinct $s,t \in \{ 1,2, \ldots ,n\}$ such that $I_s \neq R_s$ and $I_t \neq R_t$. Then choose $I' = I_1' \times 
 I_2' \times \cdots \times I_n'$ such that $I_r' = M_r$ when $r \in \{ s,t \}$ and $I_r' = R_r$ whenever $r \in \{ 1,2, \ldots, n\} \setminus \{s,t\} $. Consequently, $ I \sim I' \sim J$. Hence, $\textnormal{diam}(\textnormal{PIS}(R)) =2$. 
 \end{proof}

 \section{Strong metric dimension of prime ideal sum graphs of reduced rings}

 In this section, we obtain the strong metric dimension of the prime ideal sum graph $\textnormal{PIS}(R)$, where $R$ is a reduced ring i.e., $R \cong F_1 \times F_2 \times \cdots \times F_n$.  

\begin{lemma}\label{reducedring_nbd}
Let $R \cong F_1 \times F_2 \times \cdots \times F_n$  $(n \ge 3)$. Then for each $ I, J \in V(\textnormal{PIS}(R))$, we have $N(I) \neq N(J)$.
\end{lemma}

\begin{proof}
    Let $I = I_1 \times I_2 \times \cdots \times I_n$ and $J = J_1 \times J_2 \times \cdots \times J_n$ be any two distinct vertices of $\textnormal{PIS}(R)$. If both $I$ and $J$ are maximal ideals of $R$, then there exist distinct $s,t \in \{ 1,2, \ldots, n\}$ such that $I_s = J_t = (0)$. Without loss of generality, let $s =1$ and $j=2$ i.e., $I = (0) \times F_2 \times \cdots \times F_n$ and $J = F_1 \times (0) \times F_3 \times \cdots \times F_n$. Now let $I' = (0) \times F_2 \times (0) \times F_4 \times \cdots \times F_n$. Then $I' \sim I$ but $I' \nsim J$. It implies that $N(I) \neq N(J)$.
    
    Now without loss of generality, assume that $I$ is a maximal but $J$ is not a maximal ideal of $R$. Then there exists $ l \in \{ 1,2, \ldots, n\}$ such that $I_l = F_l$ and $J_l = (0)$. Now let $
    I'' = F_1 \times \cdots \times F_{l-1} \times (0) \times F_{l+1} \times \cdots \times F_n$. Then note that $I'' \sim J$ but $I'' \nsim I$. It follows that $N(I) \neq N(J)$. 

    We may now assume that both $I$ and $J$ are not maximal ideals of $R$. Then there exists $k \in  \{ 1,2, \ldots, n\}$ such that either $I_k = (0)$ and $ J_k = F_k$, or $I_k = F_k$ and $ J_k =  (0)$. Without loss of generality let $I_k = (0)$ and $ J_k = F_k$. Now let $ J' = F_1 \times \cdots \times F_{k-1} \times (0) \times F_{k+1} \times \cdots \times F_n$. Then note that $J' \sim I$ but $J' \nsim J$. Thus, $N(I) \neq N(J)$ for any $I,J \in V(\textnormal{PIS}(R))$.
\end{proof}

\begin{lemma} \label{reducedring_clique}
Let $R \cong F_1 \times F_2 \times \cdots \times F_n$  $(n \ge 3)$. Then $\omega(\textnormal{PIS}(R)) =n$.
\end{lemma}

\begin{proof}
    To prove the result, first we show that any clique of size $n$ in $\textnormal{PIS}(R)$ is of one of the following type:

    \begin{itemize}
        \item Type-1 : $\{ I_1, I_2, \ldots, I_n \}$, where $I_k = (0) \times F_2 \times \cdots \times F_{k-1} \times (0) \times F_{k+1} \times \cdots \times F_n$.
        \item Type-2 : $\{ I_1, I_2, \ldots, I_n \}$, where $I_1 = F_1 \times (0) \times \cdots \times (0)$ and $I_k = (0) \times F_2 \times \cdots \times F_{k-1} \times (0) \times F_{k+1} \times \cdots \times F_n$ for each $ 2 \le k \le n$.
    \end{itemize}

    Let $C = \{ I_1, I_2, \ldots, I_n \}$ be a clique of size $n$ in $\textnormal{PIS}(R)$.

    \noindent\textbf{Case-1.} If all $I_k$'s have the ideal $(0)$ at the same position. Without loss of generality assume that each $I_k$ has the ideal $(0)$ at first position i.e., $I_k= (0) \times J_2 \times J_3 \times \cdots \times J_n$,  where $J_i \in \{ (0), F_i \} $ for all $i \in \{ 2,3, \ldots, n\}$. Since $C$ is a clique in $\textnormal{PIS}(R)$, any two distinct $I_l, I_k \in C$ cannot have the ideal $(0)$ at the same position other than the first position. Since $|C|=n$, we deduce that $C$ is of Type-1.
    
    \noindent\textbf{Case-2.} If $n-1$ elements of $C$ have the ideal $(0)$ at same position. Without loss of generality, assume that each $I_k$ $(2 \le k \le n)$ has the ideal $(0)$ at the first position. Note that any two distinct ideals of these $n-1$ ideals in $C$ cannot have the ideal $(0)$ at the same position other than the first position. Since all these $n-1$ elements of $C$ forms a clique in $\textnormal{PIS}(R)$, we have 
    \begin{align*}
        I_2 &= (0) \times (0) \times F_3 \times \cdots \times F_n, \\ I_3 &= (0) \times F_2 \times (0) \times F_4 \times \cdots \times F_n,\\
        \vdots \\
        I_n &= (0) \times F_2 \times \cdots \times F_{n-1} \times (0).
    \end{align*}
     Since $I_1 \sim I_r$ for all $r \in \{ 2,3, \ldots, n\}$ and $I_1$ does not have the ideal $(0)$ at first position, we have $I_1 = F_1 \times (0) \times \cdots \times (0)$. Thus, $C$ is of Type-2. 

     \noindent\textbf{Case-3.} If $n-2$ elements of $C$ have the ideal $(0)$ at the same position. Without loss of generality, assume that each $I_k$ $(1 \le k \le n-2)$ has the ideal $(0)$ at the first position. Since all these $n-2$ elements of $C$ forms a clique in $\textnormal{PIS}(R)$, we conclude that
    \begin{align*}
        I_1 &= (0) \times (0) \times F_3 \times \cdots \times J_1, \\ I_2 &= (0) \times F_2 \times (0) \times F_4 \times \cdots \times J_2,\\
        \vdots \\
        I_{n-2} &= (0) \times F_2 \times \cdots \times F_{n-2} \times (0) \times J_{n-2},
    \end{align*}
    where one of $J_i$ $(1 \le i \le n-2)$ can be the ideal $(0)$. Without loss of generality let $J_{n-2} = (0)$ and $J_i = F_n$ for all $i \in \{1,2, \ldots, n-3\}$. Then $I_{n-2} = (0) \times F_2 \times \cdots \times F_{n-2} \times (0) \times (0)$. Since, $I_{n-1} \sim I_l$ for all $ l \in \{ 1,2, \ldots, n-2\}$, we have either $I_{n-1} = F_1 \times (0) \times \cdots \times (0) \times F_n$ or $I_{n-1} = F_1 \times (0) \times \cdots \times (0) \times F_{n-1} \times (0)$. But $ F_1 \times (0) \times \cdots \times (0) \times F_n \nsim F_1 \times (0) \times \cdots \times (0) \times F_{n-1} \times (0)$. Consequently, $|C| = n-1$, a contradiction. Similarly, If $J_i = F_n$ for each $i \in \{ 1,2, \ldots, n-2\}$, then $I_{n-2}= (0) \times F_2 \times \cdots \times F_{n-2} \times (0) \times F_n$ which gives a contradiction to the fact $|C| =n$.

    Similarly, If $n-k$ $(k \ge 3)$ elements of $C$ have the ideal $(0)$ at the same position, then we can  not get a clique of size $n$. Hence, any clique of size $n$ is either of Type-1 or Type-2.

    Now suppose that $S$ is a clique of size $t > n$ in $\textnormal{PIS}(R)$. Then there exists a clique $S'(\subset S)$ such that $|S'| =n$. Then $S'$ is of Type-1 or of Type-2. Suppose that $S' = \{  I_1, I_2, \ldots, I_n\}$ is of Type-1 and $I_t \in S \setminus S'$. Then $I_t \sim I_r$ for all $r \in \{ 1,2, \ldots, n\}$. Consequently, $I_t = (0) \times F_2 \times \cdots \times F_n$. This implies that $I_t \in S'$, a contradiction. We may now assume that $S'$ is a clique of Type-2 and $I_l \in S \setminus S'$. If $I_l$ has the ideal $(0)$ at the first position, then $I_l \sim I_r $ for each $ r \in \{ 2,3, \ldots, n \}$ gives $I_l = (0) \times F_2 \times \cdots \times F_n$. Then $I_l \nsim I_1$, a contradiction. If $I_l$ does not have $(0)$ at the first position, then $I_l \sim I_r$ for each $r \in \{ 2,3, \ldots, n \}$ gives $I_l = F_1 \times (0) \times \cdots \times (0)$. It implies that $I_l \in S'$, a contradiction.

    Thus, $\omega(\textnormal{PIS}(R))=n$.  
\end{proof}

In view of Theorems \ref{reducedgraphsdim}, \ref{daimeter} and Lemmas \ref{reducedring_nbd}, \ref{reducedring_clique}, we have the following Theorem.

\begin{theorem}\label{reducedringdimension}
 Let $R \cong F_1 \times F_2 \times \cdots \times F_n$  $(n \ge 3)$. Then $\textnormal{sdim}(\textnormal{PIS}(R)) = 2^n -n-2$.   
\end{theorem}

\section{Strong metric dimension of prime ideal sum graphs of non-reduced rings}

 In this section, we compute the strong metric dimension of the prime ideal sum graph $\textnormal{PIS}(R)$ of various classes of non-reduced rings.

\begin{lemma}\label{uniqueidealring_nbd}
Let $R \cong R_1 \times R_2 \times \cdots \times R_n$  $(n \ge 2)$, where each $R_i$ is a local ring with unique non-trivial ideal. Then for each $I, J \in V(\textnormal{PIS}(R))$, we have $N(I) \neq N(J)$.
\end{lemma}

\begin{proof}
   Let $R \cong R_1 \times R_2 \times \cdots \times R_n$  $(n \ge 2)$, where each $R_i$ is a local ring with a unique non-trivial ideal $\mathcal{M}_i$. Let $I = I_1 \times I_2 \times \cdots \times I_n$ and $J = J_1 \times J_2 \times \cdots \times J_n$ be any two distinct vertices of $\textnormal{PIS}(R)$. If both $I$ and $J$ are maximal ideals of $R$, then there exist distinct $s,t \in \{ 1,2, \ldots, n\}$ such that $I_s = \mathcal{M}_s$ and $J_t = \mathcal{M}_t$. Without loss of generality, assume that $s=1$ and $t=2$ i.e., $I = \mathcal{M}_1 \times R_2 \times \cdots \times R_n$ and $J = R_1 \times \mathcal{M}_2 \times R_3 \times \cdots \times R_n$. Then choose $I' = \mathcal{M}_1 \times R_2 \times \mathcal{M}_3 \times R_4 \times \cdots \times R_n$. Note that $I' \sim I$ but $I' \nsim J$. It implies that $N(I) \neq N(J)$.

    Now assume that $I$ is maximal, but $J$ is a non-maximal ideal of $R$. Without loss of generality let $I = \mathcal{M}_1 \times R_2 \times \cdots \times R_n$. Let $J_r = (0) $ for some $r \in \{ 1,2, \ldots, n\}$. If $r=1$, then consider $I'' = (0) \times I_2'' \times \cdots \times I_n''$ such that $I'' \neq J$. Observe that $I'' \sim I$ but $I'' \nsim J$. For $r \neq 1$, we have $ J = J_1 \times \cdots \times J_{r-1} \times (0) \times J_{r+1} \times \cdots \times J_n $. Consider $J' = (0) \times J_2' \times \cdots \times J_{r-1}' \times (0) \times J_{r+1}' \times \cdots \times J_n'$ such that $J' \neq J$. Then note that $J' \sim I$ but $J' \nsim J$. Now suppose that $J_r \neq (0)$ for each $r \in \{1,2, \ldots,n \}$, then there exists $s \in \{ 1,2, \ldots, n\}$ such that $I_s = R_s$ but $J_s = \mathcal{M}_s$. Then choose $J'' = \mathcal{M}_1 \times \cdots \times \mathcal{M}_{s-1} \times (0) \times \mathcal{M}_{s+1} \times \cdots \times \mathcal{M}_n$. Then observe that $J'' \sim I $ but $J'' \nsim J$. It follows that $N(I) \neq N(J)$.

    We may now suppose that both $I$ and $J$ are non-maximal ideals of $R$. If there exist $t \in \{ 1,2, \ldots, n \}$ such that $I_t \neq R_t$ but $J_t = R_t$, then choose $I' = R_1 \times R_2 \times \cdots \times R_{t-1} \times \mathcal{M}_t \times R_{t+1} \times \cdots \times R_n$. Then $I' \sim I$ but $I' \nsim J$. If there does not exists any $t \in \{1,2, \ldots,n \} $ such that either $I_t \neq R_t$ but $J_t = R_t$, or $ I_t = R_t$ but $J_t \neq R_t$ then there exists $ l \in \{1,2, \ldots, n \} $ such that either $I_l = (0)$ and $J_l = \mathcal{M}_l$, or $I_l = \mathcal{M}_l$ and $J_l = (0)$. Without loss of generality, assume that $I_l = (0)$ and $J_l = \mathcal{M}_l$. Then take $J' = R_1 \times \cdots \times R_{l-1} \times (0) \times R_{l+1} \times \cdots \times R_n$. Observe that $J' \sim J$ but $J' \nsim I$. Thus, the result holds. 
\end{proof}

\begin{lemma} \label{uniqueideal_clique}
Let $R \cong R_1 \times R_2 \times \cdots \times R_n$  $(n \ge 2)$, where each $R_i$ is a local ring with a unique non-trivial ideal $\mathcal{M}_i$. Then $\omega(\textnormal{PIS}(R)) =n+1$.
\end{lemma}

\begin{proof}
    First we show that any clique of size $n+1$ in $\textnormal{PIS}(R)$ is of the type
\begin{center}
    Type-1 : $\{ I_1, I_2, \ldots, I_{n+1} \}$, where $I_1 = \mathcal{M}_1 \times R_2 \times \cdots \times R_n$, $I_k = \mathcal{M}_1 \times R_2 \times \cdots \times R_{k-1} \times J_k \times R_{k+1} \times \cdots \times R_n $ $(2 \le k \le n)$ such that $J_k \in \{ (0), \mathcal{M}_k\}$ and $I_{n+1} = (0) \times R_2 \times \cdots \times R_n$.
 \end{center} 
 Let $C = \{ I_1, I_2, \ldots, I_{n+1} \}$ be a clique of size $n+1$ in $\textnormal{PIS}(R)$.

    \noindent\textbf{Case-1.} If all the elements of $C$ have the maximal ideal at the same position. Without loss of generality, assume that each $I_k$ has the ideal $\mathcal{M}_1$ at the first position. Then $I_k = \mathcal{M}_1 \times J_{k2} \times J_{k3} \times \cdots \times J_{kn}$, where $J_{kr} \in \{ (0), \mathcal{M}_r, R_r \} $ for all $ 2 \le r \le n$. If each $J_{kr} = R_r$, then $I_1 = \mathcal{M}_1 \times R_2 \times \cdots \times R_n$. Further note that if $J_{pr} \in \{ (0), \mathcal{M}_r\}$ for some $r \in \{2,3, \ldots, n \}$, then $J_{sr} = R_r$ for each $s \in \{1,2, \ldots, n+1 \} \setminus \{ p\}$. For $s \in \{ 2,3, \ldots, n+1 \} $, note that there are $n-1$ possibilities of $J_{st}$ such that $J_{st} \neq R_t$, where $2 \le t \le n$. To get the largest clique, we can choose exactly one such $J_{st}$ in each $I_s$ and $J_{sr} =  R_r$ for all $r \in \{2,3, \ldots , n \} \setminus \{ t\} $. Then the maximum possible elements in $C$ are 
    \begin{align*}
    I_1 &= \mathcal{M}_1 \times R_2 \times \cdots \times R_n \text{ and}\\
        I_k &= \mathcal{M}_1 \times R_2 \times \cdots \times R_{k-1} \times J_k \times R_{k+1} \times \cdots \times R_n (2 \le k \le n) \text{ such that } J_k \in \{ (0), \mathcal{M}_k\}.
    \end{align*}
    It follows that $|C| \le n < n+1$, a contradiction. Therefore, this case is not possible.

\noindent\textbf{Case-2.} If $n$ elements of $C$ have the maximal ideal at the same position. Without loss of generality, assume that each $I_k $ $(1 \le k \le n)$ has the ideal $\mathcal{M}_1$ at the first position. Then by a similar argument used in Case-1, the elements of $C$ are of the form
\begin{align*}
    I_1 &= \mathcal{M}_1 \times J_2 \times R_3 \times \cdots \times R_n,\\
     I_2 &= \mathcal{M}_1  \times R_2 \times J_3 \times R_4 \times \cdots \times R_n,\\
     \vdots \\
        I_{n-1} &= \mathcal{M}_1 \times R_2 \times \cdots \times R_{n-1} \times J_n,\\
        I_n &= \mathcal{M}_1 \times R_2 \times \cdots \times  R_n,
    \end{align*}
    where $J_k \in \{ (0), \mathcal{M}_k\}$ for all $k \in \{2,3, \ldots, n \}$. Note that since $I_{n+1} \sim I_n$, we have $I_{n+1} = (0) \times L_2 \times L_3 \times \cdots \times L_n $. Also $I_{n+1} \sim I_r $ for each $ r \in \{ 1,2, \ldots n-1\}$ implies that $I_{n+1} = (0) \times R_2 \times R_3 \times \cdots \times R_n$. Thus, $C$ is of Type-1.

    \noindent\textbf{Case-3.} If $n-1$ elements of $C$ have the maximal ideal at the same position. Without loss of generality, assume that each $I_k $ $(1 \le k \le n-1)$ has the ideal $\mathcal{M}_1$ at the first position. Then by a similar argument used in Case-1, the elements of $C$ are of the form
    \begin{align*}
         I_1 &= \mathcal{M}_1 \times J_2 \times R_3 \times \cdots \times R_n,\\
     I_2 &= \mathcal{M}_1  \times R_2 \times J_3 \times R_4 \times \cdots \times R_n,\\
     \vdots \\
     I_{n-2} &= \mathcal{M}_1 \times R_2 \times \cdots \times R_{n-2} \times J_{n-1} \times R_n,\\
        I_{n-1} &= \mathcal{M}_1 \times R_2 \times \cdots \times R_{n-1} \times L,
    \end{align*}
    where $J_k \in \{ (0), \mathcal{M}_k\}$ for all $ k \in \{ 2,3, \ldots, n-1 \}$ and $L \in \{ (0), \mathcal{M}_n, R_n\}$. If $L = R_n$, then $I_n$ can be one of these three ideals: $(0) \times R_2 \times \cdots \times R_n$, $(0) \times R_2 \times \cdots \times R_{n-1} \times (0)$ or $(0) \times R_2 \times \cdots \times R_{n-1} \times \mathcal{M}_n$. Note that none of these vertices are adjacent with each other. It implies that there does not exist $I_{n+1} \in V(\textnormal{PIS}(R))$ such that $I_{n+1} \in C$. Consequently, $|C| < n+1$, a contradiction. If $L \in \{ (0), \mathcal{M}_n\}$, Then $I_{n}$ can be one of these ideals: $(0) \times R_2 \times \cdots \times R_n$ or $R_1 \times L_2 \times \cdots \times L_n$ such that $L_i + J_i = \mathcal{M}_i$ for all $2 \le i \le n-1$ and $L + L_n= \mathcal{M}_n$. Note that none of these vertices are adjacent with each other. It implies that $|C|< n+1$, a contradiction. Therefore, this case is not possible.

     Similarly, if $n-k$ $(k \ge 2)$ elements of $C$ have the maximal ideal at a same position, then we can get a contradiction to the fact $ |C| = n+1$. Hence, any clique of the size $n+1$ is of Type-1.

     Now suppose that $S$ is a clique of size $t > n+1$ in $\textnormal{PIS}(R)$. Then there exists a clique $S'(\subset S)$ such that $|S'| =n+1$. Then $S'$ is of the Type-1. Let $S' = \{  I_1, I_2, \ldots, I_{n+1}\}$ and $I_t \in S \setminus S'$. Then $I_t \sim I_r $ for every $ r \in \{ 1,2, \ldots, n+1\}$. Consequently,  either $I_t = (0) \times R_2 \times \cdots \times R_n$ or $I_t = \mathcal{M}_1 \times R_2 \times \cdots \times R_n$. It implies that $I_t \in S'$, a contradiction. Thus,   $\omega(\textnormal{PIS}(R)) = n+1$.
\end{proof}

In view of Theorems \ref{reducedgraphsdim}, \ref{daimeter} and Lemmas \ref{uniqueidealring_nbd}, \ref{uniqueideal_clique}, we have the following Theorem.

\begin{theorem}\label{uniquidealdimension}
 Let $R \cong R_1 \times R_2 \times \cdots \times R_n$  $(n \ge 2)$, where each $R_i$ is a local ring with unique non-trivial ideal $\mathcal{M}_i$. Then $\textnormal{sdim}(\textnormal{PIS}(R)) = 3^n -n-3$.   
\end{theorem}

\begin{corollary}
Let $R \cong R_1 \times R_2 \times \cdots \times R_n \times F_1 \times \cdots \times F_n$ $(n,m \ge 1)$, where each $R_i$ has a unique non-trivial ideal $\mathcal{M}_i$. Then
$\text{sdim}(\textnormal{PIS}(R)) = 3^n2^m -n-3$.
\end{corollary}

\begin{remark}\label{principalideal_nbd}
Let $R \cong R_1 \times R_2 \times \cdots \times R_n$ $(n \ge2)$, where each $R_i$ is a local principal ideal ring  with maximal a ideal $\mathcal{M}_i$ such that $|\mathcal{I}^*(R_i)| \ge 2$ for every $ 1 \le i \le n$. Let $I = I_1 \times I_2 \times \cdots \times I_n$, $J = J_1 \times J_2 \times \cdots \times J_n$ $\in V(\textnormal{PIS}(R))$. Then by Lemma \ref{uniqueidealring_nbd}, observe that $N(I) = N(J)$ if and only if the following conditions hold.
\begin{enumerate}
    \item $I_r = \mathcal{M}_r$ if and only if $J_r = \mathcal{M}_r$
    \item $I_r \subsetneq \mathcal{M}_r$ if and only if $J_r \subsetneq \mathcal{M}_r$
    \item $I_r = \mathcal R_r$ if and only if $J_r = R_r$
\end{enumerate}
Also, note that for any $I,J \in  V(\textnormal{PIS}(R))$ if $N(I) = N(J)$, then $I \nsim J$. Therefore, $N[I] \neq N[J]$ for each $I,J \in  V(\textnormal{PIS}(R))$.
\end{remark}

\begin{lemma} \label{principalideal_clique}
Let $R \cong R_1 \times R_2 \times \cdots \times R_n$  $(n \ge 2)$, where each $(R_i, \mathcal{M}_i)$ is a local principal ideal ring such that $|\mathcal{I}^*(R_i)| \ge 2$ for every $ 1 \le i \le n$. Then $\omega(\textnormal{PIS}(R)) =n+1$.
\end{lemma}

\begin{proof}
    In the similar lines of the proof of Lemma \ref{uniqueideal_clique}, we can prove the desired result by showing that any clique of size $n+1$ is of the type
    \begin{center}
    $\{ I_1, I_2, \ldots, I_{n+1} \}$, where $I_1 = \mathcal{M}_1 \times R_2 \times \cdots \times R_n$, $I_k = \mathcal{M}_1 \times R_2 \times \cdots \times R_{k-1} \times J_k \times R_{k+1} \times \cdots \times R_n $ $(2 \le k \le n)$ such that $J_k \in \{ (0), \mathcal{M}_k\}$ and $I_{n+1} = L \times R_2 \times \cdots \times R_n$ with $L \in \{(0), \mathcal{M}_1^2, \mathcal{M}_1^3, \ldots, \mathcal{M}_1^{\eta_{(\mathcal{M}_1)}} \}$.
 \end{center}  
\end{proof}

In view of Theorems \ref{reducedgraphsdim}, \ref{daimeter}, Remark \ref{principalideal_nbd} and Lemma \ref{principalideal_clique}, we have the following Theorem.

\begin{theorem}
Let $R \cong R_1 \times R_2 \times \cdots \times R_n$, where each $R_i$ is a local principal ideal ring such that $|\mathcal{I}^*(R_i)| \ge 2$ for every $ 1 \le i \le n$. Then
$\text{sdim}(\textnormal{PIS}(R)) = |V(\textnormal{PIS}(R))| - n-1$.
\end{theorem}

\vspace{.3cm}
\textbf{Acknowledgement:} The first author gratefully acknowledges Birla Institute of Technology and Science (BITS) Pilani, Pilani campus, India, for providing financial support.


\end{document}